\documentclass[english]{article}

\usepackage{amsmath,amsfonts,amsthm,amssymb}
\usepackage{graphicx}
\usepackage{epstopdf}
\usepackage{enumerate}
\usepackage{booktabs}
\usepackage[latin1]{inputenc}
\usepackage{stmaryrd}
\usepackage{wasysym}
\usepackage{hyperref}
\usepackage[overload]{empheq} 
\usepackage{color}
\usepackage{float}
\usepackage{captdef}
\usepackage{fancyhdr}
\usepackage{tgschola}
\usepackage{mathptmx}

\newtheorem{theorem}{Theorem}
\newtheorem{proposition}{Proposition}
\newtheorem{remark}{Remark}
\newtheorem{lemma}{Lemma}

\newcommand{\R}{{\mathbb R}}
\newcommand{\J}{{\mathbb J}}

\title{Monotonicity properties arising in a simple model of \emph{Wolbachia} invasion for wild mosquito populations}
\author{ Diego Vicencio$^{1}$\footnote{Corresponding author: diego.vicencio@alumnos.usm.cl}, \ Olga Vasilieva$^{2}$, \  Pedro Gajardo$^{1}$ \\
$^1$ \small Universidad Técnica Federico Santa María, Valparaiso, Chile\\
$^2$ \small Universidad del Valle, Cali, Colombia \\
}

\date{\today}

\begin{document}
\maketitle
\begin{abstract}
In this paper, we propose a simplified bidimensional \textit{Wolbachia} infestation model in a population of \textit{Aedes aegypti}  mosquitoes, preserving the main features associated with the biology of this species that can be found in higher-dimensional models. Namely, our model represents the maternal transmission of the \textit{Wolbachia} symbiont, expresses the reproductive phenotype of cytoplasmic incompatibility, accounts for different fecundities and mortalities of infected and wild insects, and exhibits the bistable nature leading to the so-called \emph{principle of competitive exclusion}.  Since \textit{Wolbachia}-based biocontrol is now accepted as an ecologically friendly and potentially cost-effective method for prevention and control of dengue and other arboviral infections, it is essential to have reduced models with the main biological characteristics of \textit{Aedes aegypti} in the presence of \textit{Wolbachia}-carriers because such models help to simplify the mathematical analysis for determining appropriate biocontrol strategies. Using tools borrowed from monotone dynamical system theory, in the proposed model, we prove the existence of an invariant threshold manifold that allows us to provide practical recommendations for performing single and periodic releases of \textit{Wolbachia}-carrying mosquitoes, seeking the eventual elimination of wild insects that are capable of transmitting infections to humans. We illustrate these findings with numerical simulations using parameter values corresponding to the \textit{wMelPop} strain of \textit{Wolbachia} that is considered the best virus blocker but induces fitness loss in its carriers.

\emph{Keywords:} \textit{Wolbachia} symbiont, monotone dynamics, competitive model, minimum viable population size.
\end{abstract}

\section{Introduction}
\label{sec-intro}

It is widely known that female mosquitoes of the species \textit{Aedes aegypti} are major transmitters of dengue and other vector-borne infections. When deliberately infected with \textit{Wolbachia}, they lose their vector competence by becoming far less capable of developing a viral load sufficient for transmission of the virus to humans. Due to this remarkable feature, \textit{Wolbachia}-based biocontrol of mosquito populations has recently emerged as a novel method for the prevention and control of vector-borne infections and is accepted as an ecologically friendly and potentially cost-effective method \cite{Hoffmann2011,McMeniman2009,Moreira2009,Ruang2006,Walker2011}.

The goal of \textit{Wolbachia}-based biocontrol is the eventual elimination of wild insects (capable of transmitting the virus to humans) by performing single or periodic releases of \textit{Wolbachia}-carrying mosquitoes in some determined localities initially populated by wild mosquitoes. The practical implementation of this method requires mass-rearing of a large quantity of \textit{Wolbachia}-infected insects for posterior releases, and the desired result depends on the progressive \textit{Wolbachia} invasion and its durable establishment in wild mosquito populations. The final outcome of this process is usually referred to as ``population replacement''\cite{Almeida2019b,McMeniman2010,Ritchie2015}.

In the literature from the last decade, one can find a variety of  \textit{Wolbachia} infestation models created with the purpose, among others, of evaluating biocontrol strategies that seek suppression of wild mosquito populations. Previous studies include some frequency-based models formulated as a single equation \cite{Almeida2020,Schraiber2012,Turelli2010}, models considering only female insects \cite{Almeida2019b,Campo2017a,Campo2017b,Contreras2020,Fenton2011}, models accounting for sex structure \cite{Campo2018,Farkas2017,Ferreira2020,Zheng2014} or stage structure \cite{Adekunle2019,Almeida2019a,Bliman2018,Xue2017}, and more sophisticated models that take into account the mosquito maturation delay or their age structure \cite{Farkas2010,Ferreira2020,Zheng2014}.

The main features of all these models describing the natural dynamics of \textit{Wolbachia} are related to mimicking vertical transmission and the interference of the reproductive outcomes induced by cytoplasmic incompatibility (or CI reproductive phenotype). In this context, it is meaningful to note that CI occurs when a female, uninfected by \textit{Wolbachia}, is inseminated by an infected male, producing inviable eggs. Thus, the CI  reproductive phenotype suppresses the growth of the uninfected population and facilitates \textit{Wolbachia} spread.

Different strains of \textit{Wolbachia} may induce either perfect (100\%) or imperfect (less than 100 \%) maternal transmission and CI reproductive phenotype. Notably, some of the existent models are designed to address imperfect maternal transmission and/or imperfect CI \cite{Adekunle2019,Almeida2019a,Almeida2020,Almeida2019b,Farkas2017,Farkas2010,Ferreira2020,Xue2017}. Nevertheless, laboratory trials evince that \textit{Aedes aegypti} mosquitoes deliberately infected by the \textit{wMelPop Wolbachia} strain (regarded as the best blocker of arboviral infections \cite{Moreira2009,Dorigatti2018,Dutra2016,Ferguson2015,Walker2011,Woolfit2013}) exhibit almost perfect CI and maternal transmission \cite{Dorigatti2018,Yeap2011}. Table \ref{tab:CI} summarizes the results of matings between infected and uninfected mosquitoes when the maternal transmission and CI are perfect (100\%). On the other hand, many scholars also indicate that the \textit{wMelPop} strain is associated with high ``fitness costs'' since it reduces the female fecundity, the viability of eggs, and the lifespan of infected mosquitoes \cite{Dorigatti2018,McMeniman2010,Schraiber2012,Ritchie2015}. The latter makes the spread of \textit{wMelPop Wolbachia} infection a rather challenging task.

\begin{table}[h]
\begin{center}
	\begin{tabular}{||c||c|c||}
		\hline \hline
		\multicolumn{3}{||c||}{Mosquito offspring} \\
		\hline
		Adults & \textit{Wolbachia}-infected \female & Uninfected \female \\
		\hline \hline
		\textit{Wolbachia}-infected \mars & \textcolor{red}{\textbf{Infected}}&\textcolor{blue}{\textbf{Inviable eggs}}\\
		\hline
		Uninfected \mars & \textcolor{red}{\textbf{Infected}}&\textbf{Uninfected}\\
		\hline \hline
	\end{tabular}
	\caption{Illustration of the CI reproductive phenotype and maternal transmission of \textit{Wolbachia}. \label{tab:CI}}
\end{center}
\end{table}

The group of models assuming perfect CI and maternal transmission of \textit{Wolbachia} exhibit a bistable nature that makes them fully compliant with the so-called \emph{principle of competitive exclusion} \cite{Hsu1996}. This phenomenon implies the existence of a certain (dynamic) threshold in the current level (or frequency) of infection above which \textit{Wolbachia} is capable of invading and persisting in the uninfected population and below which the population is driven toward extinction. 
For each infection frequency, this threshold can be expressed in terms of the so-called \emph{minimal viable sizes} of each population (infected and uninfected) that are tightly related to the frequency-dependent Allee effect. Anticipated knowledge of the minimal viable size of the \textit{Wolbachia}-infected population corresponding to the current size of the wild population is the key issue for determining the appropriate size of release(s) for implementation of \textit{Wolbachia}-based biocontrol, and one of the main goals of the present work consists in assessing the release size(s) through the use of a simplified bidimensional model of \textit{Wolbachia} invasion. 

In effect, many scholars have intuitively detected the presence of the aforementioned threshold in the infection frequency (see, e.g., \cite{Almeida2019b,Bliman2018,Campo2017a,Campo2018,Contreras2020,Farkas2017,Fenton2011,Ferreira2020,Schraiber2012,Turelli2010,Xue2017,Zheng2014}). However, none of these works were focused on explicit identification of the minimal viable sizes of infected and uninfected populations nor their useful significance for the practical implementation of \textit{Wolbachia}-based biocontrol. The present paper intends to contribute to this strand of research by filling that gap. 

All models describing \textit{Wolbachia} invasion are competitive; however, only a few of them rigorously exhibit the property of monotonicity (see \cite{Almeida2019b,Bliman2018} and a simplified version without maturation delay considered in \cite{Zheng2014}). This important property makes the theoretical analysis of these models simpler by making use of the variety of research results developed for monotone dynamical systems and assembled in \cite{Smith1995}. In particular, the application of the theory of monotone systems to competitive dynamics allows identifying a partial order under which a competitive monotone system exhibits the so-called \emph{saddle-point behavior} \cite{Jiang2004}, and the latter bears a strong relationship to the principle of competitive exclusion \cite{Hsu1996}. In general terms, a system exhibiting the saddle-point behavior possesses two locally stable equilibria on the boundary and one unstable (saddle-point) equilibrium in the interior of the state domain. Moreover, the state domain is divided into three disjoint and invariant parts: two attraction basins of boundary equilibria and the so-called ``threshold manifold'' (or \emph{separatrix}) containing the unstable equilibrium that separates the attraction basins. One of the goals of the present work is to identify the threshold manifold and study its properties in light of the practical implementation of \textit{Wolbachia}-based biocotrol. 

For that purpose, we have developed a simplified bidimensional variant of the four-dimensional model presented in \cite{Bliman2018} that retains all the key properties of the original four-dimensional model, including the property of monotonicity. Notably, our reduced model bears a certain degree of similarity to the one studied in \cite{Zheng2014}, but they are not the same. The bidimensional model obtained in \cite{Zheng2014} by omitting the maturation delay assumes only the density-dependent mortality of both mosquito populations and ignores their natural mortalities which are different for infected and uninfected insects. In contrast, our bidimensional model accounts for both mortality types (natural and density-dependent). Furthermore, the authors of \cite{Zheng2014} have, in effect, identified the threshold manifold for their bidimensional model, but they did not discuss its underlying properties in light of the practical implementation of \textit{Wolbachia}-based biocontrol. 

Generally, bidimensional (or planar) dynamical systems have several important advantages compared to higher-dimensional systems. First, they concede a comprehensive visualization of the system's behavior in the phase plane that not only facilitates the conceptual theoretical analysis of the model but also provides meaningful interpretations of potential outcomes of the model. Second, for planar dynamical systems, there are numerous optimal control tools  \cite{Boscain2004,Sussmann1987a,Sussmann1987b,Sussmann1987c} that can be applied to the model proposed in this paper to obtain an analytically optimal synthesis of various optimization problems related to biocontrol purposes.  Such analytical solutions can be later tested in more complex models. Although we do not analyze optimal control problems in this work, the proposed model has been designed to serve these purposes in the future and to provide major insights regarding the evolution of both mosquito populations under the action of the biocontrol.

The paper is organized as follows. In Section \ref{sec-model}, we introduce a bidimensional \textit{Wolbachia} infestation model for populations of \textit{Aedes aegypti} mosquitoes. In Section \ref{sec-analisys}, we establish some basic properties of the introduced model. In Section \ref{sec-monot}, we prove that the reduced model retains the property of monotonicity inherent from the original four-dimensional model developed in \cite{Bliman2018}. Using this remarkable property and other tools from the theory of monotone dynamical systems, we establish the existence of an invariant threshold manifold for the proposed bidimensional model.  Finally, in Section \ref{sec-appl}, we discuss the core properties of the points located on the threshold manifold and propose their practical interpretations for performing single and periodic releases of \textit{Wolbachia}-infected mosquitoes to reach an eventual elimination of wild insects and thus to achieve the population replacement.

\section{Simplified model of \textit{Wolbachia} infestation}
\label{sec-model}

Let us consider two populations of mosquitoes, $P_N(t):= F_N(t) + M_N(t)$ and $P_W(t):= F_W(t) + M_W(t)$ present at the day $t \geq 0$ in some locality, where $F_N(t)$ and $M_N(t)$ stand, respectively, for the female and male insect that are free of\emph{ Wolbachia} symbiotic bacterium while $F_W(t)$ and $M_W(t)$ denote, respectively, the female and male insects infected with \textit{Wolbachia} bacterium.

Scientific evidence \cite{Aida2011,Styer2007} suggests that wild female and male mosquitoes are often evenly distributed; therefore, let us suppose that $N(t) := F_N(t)=M_N(t)$ for all $t \geq 0$. A similar assumption can be introduced, for the sake of simplicity, regarding \textit{Wolbachia}-carrying mosquitoes, that is, $W(t) := F_W(t)=M_W(t)$ for all $t \geq 0$. Then, the frequency of \textit{Wolbachia} infection in the total mosquito population, $\dfrac{P_W(t)}{P_N(t) + P_W(t)}$ can be determined by $N(t)$ and $W(t)$ as $\dfrac{W(t)}{N(t) + W(t)}$.

To propose a simplified model of \textit{Wolbachia} invasion, we have chosen as a starting point the stage-structured model of \textit{Wolbachia} infestation with four state variables developed by Bliman \emph{et al} \cite{Bliman2018}, and our final goal is to design a bidimensional (reduced) version of this model with similar characteristics.

Let us err on the side of caution while reducing the model's dimension and recall that a ``good'' model for describing \textit{Wolbachia} invasion must necessarily account for the following features:
\begin{enumerate}[(i)]
\item
\label{good1}
Maternal transmission of the bacterium \textit{Wolbachia} to the next generation. This feature implies that \textit{Wolbachia}-infected mosquitoes are progenies of \textit{Wolbachia}-carrying females. A wild female cannot produce \textit{Wolbachia}-infected offspring.
\item
\label{good2}
The reproductive phenotype of cytoplasmic incompatibility (CI). This feature implies that \textit{Wolbachia}-carrying females are capable of producing viable and \textit{Wolbachia}-infected offspring after mating with either wild or \textit{Wolbachia}-carrying males. On the other hand, wild females produce inviable offspring after mating with \textit{Wolbachia}-carrying males.
\item
\label{good3}
Positive invariance and well-posedness. Any model is ``biologically viable'' if its state variables are nonnegative for all $t \geq 0$ and their underlying trajectories are bounded.
\item
\label{good4}
Bistable nature. This feature expresses the so-called \emph{principle of competitive exclusion} \cite{Hsu1996} according to which only one mosquito population (either with or without \textit{Wolbachia}) should ultimately survive and persist.
\end{enumerate}
In mathematical terms, feature (\ref{good4}) implies the presence of two local attractors (boundary equilibria) and an unstable strictly positive coexistence equilibrium in between. From the biological standpoint, the last feature (\ref{good4}) is also directly related to the so-called \emph{frequency-dependent Allee effect} and implies the existence of a certain threshold in the frequency of \textit{Wolbachia} infection above which the wild population is eventually driven toward extinction, and \textit{Wolbachia} successfully invade the mosquito population. Such a threshold is usually referred to as the ``minimal viable population size'' of wild mosquitoes.

In addition to key features given by (\ref{good1})--(\ref{good4}), the four-dimensional model developed in \cite{Bliman2018} possesses another important property related to its monotonicity. Namely, the semiflow associated with that four-dimensional model is monotone and strongly order-preserving for the partial order induced by the cone $\mathcal{K}:= \mathbb{R}_{-} \times \mathbb{R}_{-} \times \mathbb{R}_{+} \times \mathbb{R}_{+}$. Notably there are several models of \textit{Wolbachia} invasion, both four-dimensional \cite{Campo2018,Contreras2020,Farkas2017} and bidimensional \cite{Campo2017a,Campo2017b}, that do not exhibit such an important property. Therefore, it is highly desirable to conserve this property in the reduced bidimensional version of the original model developed in \cite{Bliman2018}.

To reduce the four-dimensional model developed in \cite{Bliman2018} to a bidimensional scenario, the aquatic and aerial stages of each mosquito class (with and without \textit{Wolbachia}) can be ``merged'' into one state variable denoted by $N(t)$ and $W(t)$ for noninfected and \textit{Wolbachia}-infected insects, respectively. As a result, we obtain the following model without the stage structure
\begin{subequations}
\label{sys}
\begin{align}[left = \empheqlbrace\,]
\label{sys-N}
& \frac{d N}{dt} = F(N,W) := \rho_N N \left(\frac{N}{N+W} \right) - \alpha_N N - \beta_N N(N+W) \\
\label{sys-W}
& \frac{d W}{dt} = G(N,W) := \rho_W W - \alpha_W W - \beta_W W(N+W)
\end{align}
\end{subequations}
that is bidimensional and describes the time evolution of adult mosquito populations, $N(t)$ and $W(t)$. 

Let us provide a brief description of positive constant parameters included in the model \eqref{sys}:
\begin{itemize}
\item
The parameters $\rho_N$ and $\rho_W$ represent the fecundity rates of uninfected and \emph{Wolbachia}-infected insects, respectively, in the absence of competition (i.e., a  mean number of adult mosquitoes produced by one female on average per day during her lifespan).
\item
The parameters $\alpha_N$ and $\alpha_W$ refer to a natural mortality rate of uninfected and \emph{Wolbachia}-infected insects, respectively (note that $1/\alpha_N$ and $1/\alpha_W$ express the average lifespan of noninfected and \emph{Wolbachia}-infected mosquitoes).\footnote{Notably, system \eqref{sys} bears a certain degree of similarity to the one studied in \cite{Zheng2014}. However, system \eqref{sys} explicitly accounts for the natural mortalities, $\alpha_N$ and $\alpha_W$, of wild and \textit{Wolbachia}-infected mosquitoes while the model in \cite{Zheng2014} ignores them and only considers the density-dependent mortalities (expressed by the last terms in both equations of \eqref{sys}).}
\item
The parameters $\beta_N$ and $\beta_W$ are associated with the competition between two mosquito populations for food resources, breeding sites, and mating opportunities, including larvae development into adults under density dependence.
\end{itemize}
To guarantee survival and persistence of each population in the absence of another, the following conditions are introduced:
\begin{equation}
\label{survival}
\rho_N > \alpha_N, \qquad  \rho_W > \alpha_W.
\end{equation}

It is easy to verify that condition (\ref{good1}) on p. \pageref{good1} is fulfilled by the equation \eqref{sys-W} according to which \textit{Wolbachia}-carriers are progenies of $W$ only (cf. the positive term in the right-hand side of \eqref{sys-W} that expresses the recruitment of $W$). Additionally, equation \eqref{sys-N} captures the condition (\ref{good2}) on p. \pageref{good2} that refers to the CI-phenotype in the sense that recruitment of wild mosquitoes (cf. the positive term in the right-hand side of \eqref{sys-W}) is proportional to the number of matings between wild males and females. These outcomes agree with the description of maternal transmission and the reproductive phenotype of cytoplasmic incompatibility presented in Table \ref{tab:CI}.

To illustrate the remaining features,  a deeper analysis of the model \eqref{sys} is required, and the latter is presented in the next section.

\section{Primary properties of the reduced model}
\label{sec-analisys}

This section is focused on establishing the key features (\ref{good3}) and (\ref{good4}) on p. \pageref{good4} for the reduced bidimensional model \eqref{sys}. The well-posedness of the model is attested in Subsection \ref{subsec-wp}, and the stability analysis of the system \eqref{sys} is carried out in Subsection \ref{subsec-stability}.

\subsection{Well-posedness of the model}
\label{subsec-wp}

To verify the condition (\ref{good3}) on p. \pageref{good3}, we establish and prove the following result.

\begin{proposition}
\label{prop1}
With condition \eqref{survival} in force, dynamical system \eqref{sys} is invariant in the positive cone $\mathbb{R}_+^2$. Additionally, there exists a compact set $\mathcal{X} \subset \mathbb{R}_+^2$ such that the trajectories of \eqref{sys} engendered by initial conditions $\big( N(0),W(0) \big) \in \mathcal{X}$ remain in $\mathcal{X}$ for all $t \geq 0$.
\end{proposition}

\begin{proof}
It is immediately noted that
\[ \left. \frac{d N}{dt} \right|_{N=0} = 0 \qquad \text{and} \qquad  \left. \frac{d W}{dt} \right|_{W=0} = 0 \]
which plainly indicates that system \eqref{sys} is invariant in the positive cone $\mathbb{R}^2_+$, i.e., its trajectories $N(t)$ and $W(t)$ engendered by $\big( N(0),W(0) \big) \in \mathbb{R}_+^2$ remain in $\mathbb{R}_+^2$ for all $t \geq 0$, and thus
\[ N(t)\geq 0 \quad \text{and} \quad W(t)\geq 0 \quad \text{ for all } \quad t \geq 0. \]

Furthermore, let us observe that for $F(N,W)$ and $G(N,W)$ the following hold:
\begin{equation*}
F(N,W)\leq \beta_N N \left( \frac{\rho_N  - \alpha_N}{\beta_N} - N \right)  \qquad \text{and} \qquad G(N,W)  \leq \beta_W W\left( \frac{\rho_W - \alpha_W}{\beta_W} -W \right).
\end{equation*}
Thus, one has
\[ \left. \frac{d N}{dt} \right|_{N=\frac{\rho_N  - \alpha_N}{\beta_N} } \le 0 \qquad \text{and} \qquad  \left. \frac{d W}{dt} \right|_{W=\frac{\rho_W - \alpha_W}{\beta_W} } \le 0 ,\]
where these quantities are strictly negative when $N > \frac{\rho_N  - \alpha_N}{\beta_N}$ and $W > \frac{\rho_W - \alpha_W}{\beta_W}$.
Hence, it follows  that
\[ N(t) \leq \max \big\{ N_{\sharp}, N(0) \big\} \quad \text{and} \quad W(t) \leq \max \big\{ W_{\sharp}, W(0) \big\}  \]
where
\begin{equation}
\label{NW-sharp}
N_{\sharp} := \frac{\rho_N-\alpha_N}{\beta_N} \quad \text{and} \quad W_{\sharp}:= \frac{\rho_W-\alpha_W}{\beta_W}.
\end{equation}
In other words, the compact set
\begin{equation}
\label{ab-set}
\mathcal{X} := \Big\{ (N,W) \in \mathbb{R}_+^2: \ 0 \leq N \leq N_{\sharp}, \  0 \leq W \leq W_{\sharp}\Big\}
\end{equation}
is invariant in the sense that all trajectories $(N(t), W(t))$ of \eqref{sys} engendered by $\big( N(0),W(0) \big) \in \mathcal{X}$ remain in $\mathcal{X}$ for all $t \geq 0$.
\end{proof}

It stems from the proof of Proposition \ref{prop1} that the reduced dynamical system \eqref{sys} is \emph{dissipative} \cite{Smith1995}, and $\mathcal{X}$ is referred to as an \emph{absorbing set}. This means that all trajectories  $N(t), W(t)$ of \eqref{sys} engendered by $\big( N(0),W(0) \big) \in \mathbb{R}_+^2 \setminus \mathcal{X}$ are attracted to $\mathcal{X}$ and there is a finite $\bar{t} >0$ such that
\[ \big( N(\bar{t}),W(\bar{t}) \big) \in  \mathcal{X} \quad \text{for all } \ t \geq \bar{t}. \]

Thus, we can conclude that all solutions of \eqref{sys} engendered by nonnegative initial conditions $\big( N(0),W(0) \big) \in \R_+^2$ are uniformly ultimately bounded.

\begin{remark}
\label{rem1}
It is easy to check that $\mathbb{R}_+^2$ contains three subsets that are invariant with respect to solutions of the system \eqref{sys}:
\begin{enumerate}
  \item
  Set $\Omega_0 := \big\{ (0,0) \big\}$ containing the origin is invariant since for $\big(N(0),W(0)\big) \in \Omega_0$ we have that $\big(N(t),W(t)\big) \in \Omega_0$ for all $t \geq 0$.
  \item
  Set $\Omega_N := \big\{ (N,W) \in \mathbb{R}_+^2: \  N>0,  W=0 \big\}$ that contains the $N$-axis is invariant. In effect, $W(0)=0$ implies the absence of  \textit{Wolbachia}-carriers for all $t \geq 0$ and the dynamical system \eqref{sys} turns into the logistic equation for the wild population:
\begin{equation}
\label{log-N}
\frac{d N}{dt} = N (\rho_N - \alpha_N  - \beta_N N),
\end{equation}
whose solutions engendered by $N(0)>0$ tend to the carrying capacity $N_{\sharp}$ as $t \to \infty$. Therefore, if $\big( N(0), W(0) \big) \in \Omega_N$, we have that $\big( N(t),W(t) \big) \in \Omega_N$ for all $t \geq 0$.
  \item
  Set $\Omega_W := \big\{ (N,W) \in \mathbb{R}_+^2: \  N=0,  W>0 \big\}$ that contains the $W$-axis is also invariant because $N(0)=0$ implies the absence of wild mosquitoes for all $t \geq 0$ and the dynamical system \eqref{sys} turns into the logistic equation for the \textit{Wolbachia}-carrying population
 \begin{equation}
\label{log-W}
 \frac{d W}{dt} = W (\rho_W - \alpha_W  - \beta_W W),
\end{equation}
 whose solutions engendered by $W(0)>0$ tend to the carrying capacity $W_{\sharp}$ as $t \to \infty$. Therefore, if $\big( N(0),W(0) \big) \in \Omega_W$, we have that $\big( N(t),W(t) \big) \in \Omega_W$ for all $t \geq 0$.
\end{enumerate}
\end{remark}

\subsection{Stability analysis of the model}
\label{subsec-stability}

To verify the condition (\ref{good4}) on p. \pageref{good4}, it is necessary to determine all possible equilibria of the system \eqref{sys} that are solutions of the algebraic system
\begin{subequations}
\label{sys-fp}
\begin{align}[left = \empheqlbrace\,]
\label{sys-fp-N}
0 &= F(N,W)= \rho_N N \left(\frac{N}{N+W} \right) - \alpha_N N - \beta_N N(N+W) \\
\label{sys-fp-W}
0 &= G(N,W)= \rho_W W - \alpha_W W - \beta_W W(N+W) .
\end{align}
\end{subequations}

After some manipulations, we determine that the dynamical system \eqref{sys} has four steady states (or equilibria), all of them are located in $\mathcal{X} \subset \mathbb{R}_+^2$, defined in \eqref{ab-set}, and their coordinates can be explicitly expressed in terms of the model's parameters.
\begin{itemize}
\item
Trivial steady state $\mathbf{E}_0=(0,0)$ that corresponds to the extinction of both mosquito populations.
\item
Boundary steady state $\mathbf{E}_N=(N_{\sharp},0)$ that corresponds to the survival and persistence of the wild mosquito population and eventual extinction of the \textit{Wolbachia}-carrying population.
\item
Boundary steady state $\mathbf{E}_W=(0,W_{\sharp})$ that corresponds to the survival and persistence of the \textit{Wolbachia}-carrying population and eventual extinction of the wild mosquito population.
\item
Strictly positive steady state $\mathbf{E}_c=(N_c,W_c)$ that corresponds to the coexistence of both mosquito populations and where
\begin{equation}
\label{coex}
N_c = W_{\sharp} \left(1 - \frac{\beta_N}{\rho_N}(N_{\sharp} - W_{\sharp}) \right), \quad W_c = W_{\sharp} \frac{\beta_N}{\rho_N}(N_{\sharp} - W_{\sharp}) \quad \text{and} \quad N_c + W_c =W_{\sharp}.
\end{equation}
\end{itemize}

Figure \ref{fig-phase1} displays the abovementioned equilibria on the phase portrait of the system \eqref{sys}. This figure also shows two nullclines or zero-growth isoclines of the system: the $N$-nullcline (i.e., the curve $F(N,W)=0$) is plotted in blue color and the $W$-nullcline (i.e., the curve $G(N,W)=0$) is plotted in red color.

Notably, violation of conditions \eqref{survival} (i.e., a situation with $\rho_N < \alpha_N$ or $\rho_W < \alpha_W$) represents the case in which, for each population, the mortality rate is higher than the recruitment rate. The only outcome in such a case would be the extinction of both populations, and $E_0=(0,0)$ would be a globally asymptotically stable equilibrium, while the other equilibria ($\mathbf{E}_N, \mathbf{E}_W$, and $\mathbf{E}_c$) would become unfeasible (that is, with negative coordinates). We will not take into consideration this case, and from now on, we will assume that both conditions in \eqref{survival} are always in force.

On the other hand, the existence of a strictly positive equilibrium $\mathbf{E}_c=(N_c,W_c)$ requires that
\begin{equation}
\label{cond-coex}
N_{\sharp} - W_{\sharp} >0.
\end{equation}
From the biological standpoint, condition \eqref{cond-coex} is rather credible. Let us recall that \textit{Wolbachia} infection negatively affects the individual fitness of its carriers by reducing the females' fecundity and increasing the natural mortality of mosquitoes (see the exhaustive review by Dorigatti\emph{ et al.}, 2018 \cite{Dorigatti2018} and more precise references therein). Therefore, it holds that
\[ \rho_W < \rho_N \quad \text{and} \quad \alpha_W > \alpha_N. \]
Furthermore, several recent studies have determined that, at high levels of intraspecific competition, \textit{Wolbachia}-infected larvae experience reduced survival \cite{Crain2011,Oliveira2017,Suh2013}. In other words, the \textit{Wolbachia}-carrying insects exhibit higher mortality due to intraspecific competition than wild insects, meaning that $\beta_N \leq \beta_W$. From the above rationale, it follows that
\[ \frac{\rho_N - \alpha_N}{\beta_N} = N_{\sharp} > W_{\sharp} = \frac{\rho_W - \alpha_W}{\beta_W} \]
which is equivalent to \eqref{cond-coex}.

The following result is claimed in order to complete the proof of the condition (\ref{good4}) on p. \pageref{good4}.
\begin{proposition}
\label{prop2}
With conditions \eqref{survival} and \eqref{cond-coex} in force, the steady states $\mathbf{E}_N=(N_{\sharp},0)$ and $\mathbf{E}_W=(0,W_{\sharp})$ are asymptotically stable while the coexistence equilibrium $\mathbf{E}_c=(N_c,W_c)$ is unstable. Furthermore, the trivial steady state $\mathbf{E}_0$ is a source (nodal repeller).
\end{proposition}

\begin{proof}
Local stability of each equilibrium can be determined by the signs of eigenvalues of the Jacobian matrix associated with the system \eqref{sys}:
\begin{equation}
\label{Jacobiano}
\mathbb{J} (N,W) := \left[ \begin{matrix} \rho_N\left(1-\dfrac{W^2}{(N+W)^2}\right) - \alpha_N - \beta_N(W+2N) & -N\left(\beta_N + \rho_N\dfrac{N}{(N+W)^2} \right)\\
& \\
-\beta_W W & \rho_W - \alpha_W - \beta_W(N+2W)
\end{matrix}\right] .
\end{equation}
Direct evaluation of \eqref{Jacobiano} in the boundary steady state $\mathbf{E}_N=(N_{\sharp},0)$  renders that
\[ \mathbb{J} (N_{\sharp},0)= \left[\begin{matrix} \rho_N - \alpha_N - 2\beta_N N_{\sharp} & -N_{\sharp}\left(\beta_N + \rho_N \right)\\ & \\
0 & \rho_W - \alpha_W - \beta_W N_{\sharp}
\end{matrix}\right], \]
which is an upper-triangular matrix and its eigenvalues $\lambda_i^N, i=1,2$ are located on the main diagonal. In effect,
\[ \lambda_1^N = \rho_N - \alpha_N - 2\beta_N N_{\sharp} =- (\rho_N - \alpha_N) < 0 , \qquad \lambda_2^N =  \rho_W - \alpha_W - \beta_W N_{\sharp}=- \beta_W (N_{\sharp} - W_{\sharp}) < 0 \]
are both negative due to the conditions \eqref{survival}, \eqref{cond-coex}. Therefore, $\mathbf{E}_N=(N_{\sharp},0)$ is locally asymptotically stable (nodal attractor).

Furthermore, by direct substitution of the boundary steady state $\mathbf{E}_W=(0,W_{\sharp})$ in \eqref{Jacobiano}, we obtain 
\[ \mathbb{J} (0, W_{\sharp})= \left[\begin{matrix} - \alpha_N - \beta_N W_{\sharp} & 0 \\ & \\
-\beta_W W_{\sharp} & \rho_W - \alpha_W - 2\beta_W W_{\sharp} \end{matrix}\right] \]
which is a lower-triangular matrix and its eigenvalues $\lambda_i^W, i=1,2$ are located on the main diagonal. Consequently,
\[  \lambda_1^W = - \alpha_N - \beta_N W_{\sharp}< 0, \qquad  \lambda_2^W =  \rho_W - \alpha_W - 2\beta_W W_{\sharp} =- (\rho_W - \alpha_W) < 0 \]
are both negative by virtue of \eqref{survival}.  Therefore, $\mathbf{E}_W=(0,W_{\sharp})$ is locally asymptotically stable (nodal attractor).

On the other hand, direct evaluation of the Jacobian matrix \eqref{Jacobiano} in the coexistence steady state $\mathbf{E}_c=(N_c,W_c)$ defined by \eqref{coex} results in a rather cumbersome approach.
 However, let us recall that $\det \Big(\mathbb{J}(N_c,W_c) \Big)= \lambda_1^c \lambda_2^c$ where $\lambda_i^c, i=1,2$ denote two eigenvalues of $\mathbb{J}(N_c,W_c)$. Therefore, to prove that the coexistence equilibrium $\mathbf{E}_c=(N_c,W_c)$ is unstable, it is sufficient to show that  $ \det \Big( \mathbb{J} (N_c, W_c) \Big)$ is strictly negative. In effect,
\begin{align*}
\det \Big( \J (N_c,W_c) \Big) &= \left[ \rho_N - \alpha_N - \rho_N \frac{W_c^2}{(N_c+W_c)^2} - \beta_N(W_c+2N_c) \right] \cdot \Big[ \rho_W - \alpha_W - \beta_W (N_c+2W_c) \Big] \\  &-\beta_W W_c N_c \left[ \beta_N + \rho_N \frac{N_c}{(N_c+W_c)^2} \right] \qquad \Leftarrow \left\{ \begin{array}{l} N_c + W_c=W_{\sharp} \\ \rho_N - \alpha_N = \beta_N N_{\sharp} \\ \rho_W - \alpha_W = \beta_W W_{\sharp} \end{array} \right\} \\
&= - \beta_W W_c \left[ \beta_N(N_{\sharp} - W_{\sharp} - N_c) - \rho_N \left( \frac{W_c}{W_{\sharp}} \right)^2 \right] - \beta_W W_c \left[ \beta_N N_c + \rho_N \left( \frac{N_c}{W_{\sharp}} \right)^2 \right] \\
&=-\beta_W W_c \left[ \beta_N (N_{\sharp} - W_{\sharp}) + \frac{\rho_N}{W_{\sharp}^2} \big( N_c^2 - W_c^2 \big)\right] \qquad \Leftarrow \left\{ \rho_N = \beta_N \frac{W_{\sharp}}{W_c} (N_{\sharp} - W_{\sharp}) \right\} \\
&= - \beta_W \beta_N \Big[  W_c (N_{\sharp} - W_{\sharp}) + (N_{\sharp} - W_{\sharp}) (N_c - W_c)\Big] = - \beta_W \beta_N (N_{\sharp} - W_{\sharp}) N_c < 0.
\end{align*}
Thus, we have $\det \Big( \J (N_c,W_c) \Big) = \lambda_1^c \lambda_2^c < 0$ meaning that $\lambda_1^c$ and $\lambda_2^c$ have opposite signs. Therefore, the coexistence equilibrium $\mathbf{E}_c=(N_c,W_c)$ is unstable (saddle point).

To show that $\mathbf{E}_0=(0,0)$ is repelling, we cannot merely substitute its coordinates in the Jacobian matrix \eqref{Jacobiano} since $\mathbb{J}(0,0)$ cannot be directly computed at this steady state. Let us now recall (see Remark \ref{rem1}) that the origin is a self-contained invariant set and that the system trajectories engendered by $W(0) =0$ and $N(0)>0$ (which can be arbitrarily small) are attracted to $\mathbf{E}_N$, while the trajectories engendered by $N(0) =0$ and $W(0)>0$ (which can also be arbitrarily small) are attracted to $\mathbf{E}_W$. Thus, every vicinity of $\mathbf{E}_0$ contains initial conditions from which the system trajectories move away from the origin $\mathbf{E}_0=(0,0)$ in the direction of either $\mathbf{E}_N$ or $\mathbf{E}_W$. This means that $\mathbf{E}_W$ is a source. This completes the proof of Proposition \ref{prop2} and also establishes the validity of the condition (\ref{good4}) presented on p. \pageref{good4}.
\end{proof}

In the next section, we explore monotonicity and other important properties of the reduced system \eqref{sys}.

\section{Monotonicity and other important properties of the model}
\label{sec-monot}

To show that a certain dynamical system is monotone, it should be recalled that a dynamical system is called monotone when the flow generated by this system preserves some partial order \cite{Smith1995}. Therefore, first, we define the partial order induced by a convex cone $\mathcal{K} \subset \mathbb{R}^2$ which is typically a quadrant of $\mathbb{R}^2$ for bidimensional dynamical systems \cite{Hsu1996,Smith2017}. Given two elements $\mathbf{X}=(x_1,x_2) \in \mathcal{K}$ and $\mathbf{Y}=(y_1,y_2) \in \mathcal{K}$, we write:
\begin{itemize}
  \item $\mathbf{X} \leq_{\mathcal{K}} \mathbf{Y}$ if $\mathbf{Y} - \mathbf{X} \in \mathcal{K}$;
  \item $\mathbf{X} <_{\mathcal{K}} \mathbf{Y}$ if $\mathbf{Y} - \mathbf{X} \in \mathcal{K}\setminus {\mathbf{0}}$;
  \item $\mathbf{X} \ll_{\mathcal{K}} \mathbf{Y}$ if $\mathbf{Y} - \mathbf{X} \in \text{Int }\mathcal{K}$.
\end{itemize}
Partial order defined by the last two items is also referred to as ``strict order'' and ``strong order'', respectively.

The sets
\[ \rrbracket \mathbf{X}, \mathbf{Y}\llbracket_{\mathcal{K}} := \Big\{ \mathbf{Z} \in \mathbb{R}_+^2: \  \mathbf{X}  <_{\mathcal{K}} \mathbf{Z}  <_{\mathcal{K}} \mathbf{Y} \Big\}, \qquad \llbracket \mathbf{X}, \mathbf{Y} \rrbracket_{\mathcal{K}} := \Big\{ \mathbf{Z} \in \mathbb{R}_+^2: \  \mathbf{X}  \leq_{\mathcal{K}} \mathbf{Z}  \leq_{\mathcal{K}} \mathbf{Y} \Big\} \]
are referred to as \emph{order intervals} (open and closed, respectively) induced by the cone $\mathcal{K}$.

Let $\mathbf{X} \in \mathbb{R}_+^2$ such that $\mathbf{X}=(x_1,x_2):=(N,W)$ defines the state vector of the system \eqref{sys}. Denote as $\mathbf{F}: \mathbb{R}_+^2 \mapsto \mathbb{R}_+^2, \mathbf{F}:= \big( F(\mathbf{X}), G(\mathbf{X}) \big)$ the vector field whose components $F$ and $G$ are defined in \eqref{sys} as scalar functions of $\mathbf{X}=(N,W)$. Using these notations, our system \eqref{sys} can be written as
\begin{equation}
\label{sysF}
 \frac{d \mathbf{X}}{dt} = \mathbf{F}(\mathbf{X})
\end{equation}
and its solution, engendered by an initial condition $\mathbf{X}_0 = \big( N(0), W(0) \big) \in \mathbb{R}_+^2$, can be denoted as $\Phi(t,\mathbf{X}_0 )$.

The positive semiflow of the system \eqref{sys} or \eqref{sysF} generated by the vector field $\mathbf{F}(\mathbf{X})$ is then the continuous mapping defined by $\Phi: \ \R_+ \times \R_+^2 \mapsto \times \R_+^2$, where $\Phi(t,\mathbf{X} )$ denotes the solution of \eqref{sysF} that satisfies $\mathbf{X}(0)=\mathbf{X}$. Here, we consider only the positive semiflow of the system \eqref{sysF}, that is, with $t \in [0, \infty)$, since we are interested in the system's behavior in forward time.

The semiflow $\Phi(t,\cdot)$ is called monotone (resp. strongly monotone) on a subset $A \subset \R^2_+$, with respect to partial order induced by the cone $\mathcal{K}$ if
\[ \Phi (t,\mathbf{X}) \leq_{\mathcal{K}} \Phi(t,\mathbf{Y}) \quad \Big( \text{resp.} \ \Phi (t,\mathbf{X}) \ll_{\mathcal{K}} \Phi (t,\mathbf{Y}) \ \Big) \quad \text{whenever} \quad \mathbf{X},~\mathbf{Y} \in A; ~~\mathbf{X} \leq_{\mathcal{K}} \mathbf{Y} ~~~ \Big( \text{resp.} \ \mathbf{X} <_{\mathcal{K}} \mathbf{Y} \ \Big) \quad \text{for all } \ t \geq 0. \]

Furthermore, the semiflow $\Phi(t,\cdot)$ is called \emph{strongly order-preserving} on $A \subset \R^2_+$ (or SOP, for briefness) if $\Phi (t,\cdot)$ is monotone on $A$ and, whenever $\mathbf{X} <_{\mathcal{K}} \mathbf{Y}$, there exist open neighborhoods  $\mathcal{V}_X$ of $\mathbf{X}$ and $\mathcal{V}_Y$ of $\mathbf{Y}$ such that
\[ \Phi(t,\mathcal{V}_X) \leq_{\mathcal{K}} \Phi (t,\mathcal{V}_Y) \quad \text{for } \ t >0. \]
In other words, the partial order induced by the cone $\mathcal{K}$ is preserved for every ordered pair $\hat{\mathbf{X}} \leq_{\mathcal{K}}  \hat{\mathbf{Y}}$, where $\hat{\mathbf{X}} \in \mathcal{V}_X,  \hat{\mathbf{Y}} \in \mathcal{V}_Y$ within these open vicinities $\mathcal{V}_X$ and $\mathcal{V}_Y$ for all $t >0.$

\begin{proposition}
\label{prop-SOP}
System \eqref{sys}, \eqref{sysF} is monotone in $\mathbb{R}_+^2$ and strongly order-preserving (SOP) in the interior of $\mathbb{R}_+^2$ with respect to the partial order induced by the cone $\mathcal{K}:=\mathbb{R}_+ \times \mathbb{R}_-$.
\end{proposition}
\begin{proof}
It is easy to see that the partial order induced by $\mathcal{K}:=\mathbb{R}_+ \times \mathbb{R}_-$ is related to the ``standard'' order (induced by $\mathbb{R}_+^2$) in the following sense. For any two elements $\mathbf{X}=(x_1,x_2) \in \mathbb{R}_+^2$ and $\mathbf{Y}=(y_1,y_2) \in \mathbb{R}_+^2$, it is fulfilled that
\[ \mathbf{X} \leq_{\mathcal{K}} \mathbf{Y} \quad \text{if and only if} \quad x_1 \leq y_1 \quad \text{and} \quad x_2 \geq y_2. \]
A similar statement also holds with ``$\ll_{\mathcal{K}}$'' replacing ``$\leq_{\mathcal{K}}$'' and ``$<$'' replacing ``$\leq$''.

Using the idea of \emph{order isomorphism} (for more details, see \cite[Section 3.5]{Smith1995}) induced by the diagonal matrix
\[ \mathbb{P}:= \left[ \begin{matrix}1 & 0 \\  0 & -1 \end{matrix} \right], \qquad \mathbb{P}=\mathbb{P}^{-1}=\mathbb{P}^{T} \]
it is immediate to conclude that
\[ \mathbf{X} \leq_{\mathcal{K}} \mathbf{Y} \quad \text{if and only if} \quad \mathbb{P} \cdot \mathbf{X}  \leq \mathbb{P} \cdot \mathbf{Y}. \]
According to \cite{Angeli2003,Smith1995}, a simple change of variables
\[ \mathbf{Z} := \mathbb{P} \cdot \mathbf{X} = \mathbb{P} \cdot \left[ \begin{array}{c}  N \\ W \end{array} \right] = \left[ \begin{array}{r}  N \\ -W \end{array} \right] \quad \text{and} \quad  \mathbb{P} \cdot \mathbf{F} (\mathbf{X}) = \mathbb{P} \cdot \mathbf{F} ( \mathbb{P} \cdot \mathbf{Z}) \]
leads to the \emph{cooperative} dynamical system
\begin{equation}
\label{sysZ}
\frac{d \mathbf{Z}}{dt} = \mathbb{P} \cdot \mathbf{F} (\mathbb{P} \cdot \mathbf{Z})
\end{equation}
which is monotone with respect to the ``standard'' order (induced by $\mathbb{R}_+^2$). Furthermore, the Jacobian of this system $\mathbb{J}(\mathbf{Z})$ is linked to the Jacobian $\mathbb{J} (N,W)$ of \eqref{sys}, \eqref{sysF} by the relationship
\[ \mathbb{J}(\mathbf{Z}) = \mathbb{P} \cdot \mathbb{J} (\mathbf{X})  \cdot \mathbb{P} \]
and is a Metzler matrix
\begin{equation}
\label{JacobianoZ}
\mathbb{P} \cdot \mathbb{J} (N, W)  \cdot \mathbb{P} = \left[ \begin{matrix} \rho_N\left(1-\dfrac{W^2}{(N+W)^2}\right) - \alpha_N - \beta_N(W+2N) & N\left(\beta_N + \rho_N\dfrac{N}{(N+W)^2} \right)\\
& \\
\beta_W W & \rho_W - \alpha_W - \beta_W(N+2W)
\end{matrix}\right]
\end{equation}
for all $(N,W) \in \mathbb{R}_+^2$. The latter implies that the semiflow generated by \eqref{sysZ} preserves the ``standard'' order (induced by $\mathbb{R}_+^2$), while the semiflow $\Phi(t,\cdot)$ generated by the system \eqref{sys}, \eqref{sysF} preserves the partial order induced by the cone $\mathcal{K}:=\mathbb{R}_+ \times \mathbb{R}_-$ for all $(N,W) \in \mathbb{R}_+^2$. Moreover, the Jacobian matrix \eqref{JacobianoZ} is irreducible whenever $N \neq 0, W \neq 0.$ Therefore, the system \eqref{sys}, \eqref{sysF} is strongly order-preserving in the interior of $\mathbb{R}_+^2$ in accordance with the results from \cite[Chapter 4]{Smith1995}.
\end{proof}

Let us recall the \emph{basin of attraction} of a locally asymptotically stable equilibrium $\mathbf{X}^{*}$ for the dynamical system \eqref{sysF} is the set of initial conditions $\mathbf{X}_0$ such that the solutions $\Phi\big( t; \mathbf{X}_0 \big)$ engendered by $\mathbf{X}_0$ converge to $\mathbf{X}^{*}$ as $t \to \infty$. According to Proposition \ref{prop2}, our system \eqref{sys} possesses two local attractors, $\mathbf{E}_N=\big( N_{\sharp}, 0 \big)$ and $\mathbf{E}_W=\big( 0, W_{\sharp} \big)$, whose respective basins of attraction can be written as
\begin{subequations}
\label{basins}
\begin{align}
\label{basinN}
\mathcal{B}_{\mathbf{E}_N} & := \Big\{ (N,W) \in \mathbb{R}_+^2: \ \lim \limits_{t \to \infty} \big( N(t), W(t) \big) = \big( N_{\sharp}, 0 \big) \Big\} \\
\label{basinW}
\mathcal{B}_{\mathbf{E}_W} & := \Big\{ (N,W) \in \mathbb{R}_+^2: \ \lim \limits_{t \to \infty} \big( N(t), W(t) \big) = \big( 0, W_{\sharp} \big) \Big\}
\end{align}
\end{subequations}
where the limits are understood in the ``componentwise'' sense.

To determine sets included in the above basins of attraction, we introduce the following lemma obtained from Proposition \ref{prop-SOP}.

\begin{lemma}\label{lemma:sop-interval}
Systems \eqref{sys} and \eqref{sysF} are  strongly order-preserving (SOP) on $ \llbracket \mathbf{E}_N, \mathbf{E}_W \rrbracket_{\mathcal{K}}$ with respect to the partial order induced by the cone $\mathcal{K}:=\mathbb{R}_+ \times \mathbb{R}_-$.
\end{lemma}
\begin{proof}
We follow the proof scheme used in \cite[Theorem 5]{Bliman2018}, using the fact that strong monotonicity implies the SOP property, as indicated in  \cite[Proposition 1.1.1]{Smith1995}. First, from Proposition \ref{prop1}, it can be easily deduced that systems \eqref{sys} and \eqref{sysF} are positively invariant in the interior of $\mathbb{R}_+^2$. Furthermore, from Remark \ref{rem1}, the sets $\mathbb{R}_+\times \{0\}$, $\{0\} \times \mathbb{R}_+$ and $\{0\}\times\{0\}$ are invariant as well. Proposition \ref{prop-SOP} proves that the flow is SOP in the interior of $\mathbb{R}_+^2$ which is invariant. In each invariant set $\mathbb{R}_+\times \{0\}$ and $\{0\} \times \mathbb{R}_+$, systems \eqref{sys} and \eqref{sysF} are reduced to a quadratic equation, equivalent to a logistic equation growth in each case. In each of these sets, in the order relation restricted to each invariant set, the flow is strongly monotone, and thus, the flow is also SOP in these sets. The condition holds trivially in $\{0\}\times\{0\}$ given that this point is a steady state.

Finally, having proven that the flow is SOP for initial conditions in each of the invariant sets, from \cite[Remark 5.1.1]{Smith1995}, we can deduce that given an initial condition in one of these sets (i.e., in the border of $\R^2_+$) and another initial condition in the interior of $\mathbb{R}_+^2$, strong monotonicity is also preserved, which concludes the proof that the flow is SOP on $ \llbracket \mathbf{E}_N, \mathbf{E}_W \rrbracket_{\mathcal{K}}$ because this interval is a subset of $\R^2_+$.
\end{proof}

The combination of results established by Propositions \ref{prop2} and \ref{prop-SOP} and Lemma \ref{lemma:sop-interval} together provide an essential basis for applying the fundamental \emph{Order Interval Trichotomy Theorem} (see, e.g., \cite[Theorem 2.2.2]{Smith1995}) to the order interval $ \llbracket \mathbf{E}_N, \mathbf{E}_W \rrbracket_{\mathcal{K}}$ and two subintervals it contains. Let us recall the formulation of this theorem.
\begin{theorem}[\cite{Smith1995}, Theorem 2.2.2]
\label{teor1}
Let the semiflow $\Phi(t,\cdot)$ of the dynamical system given in general form \eqref{sysF} be SOP with respect to the partial order induced by some cone $\mathcal{C}$ on the order interval $\llbracket \mathbf{X}^{*}, \mathbf{Y}^{*}\rrbracket_{\mathcal{C}}$ where $ \mathbf{X}^{*} <_{\mathcal{C}} \mathbf{Y}^{*}$ and $ \mathbf{X}^{*}, \mathbf{Y}^{*}$ are equilibria of \eqref{sysF}. If $\overline{\Phi \big(t, \llbracket \mathbf{X}^{*}, \mathbf{Y}^{*}\rrbracket_{\mathcal{C}} \big)}$ is compact for each $t>0$, then one of the following holds:
\begin{description}
  \item[(i)]
  There exists another equilibrium  $\mathbf{Z}^{*} \in \rrbracket \mathbf{X}^{*}, \mathbf{Y}^{*}\llbracket_{\mathcal{C}}$ of the system \eqref{sysF}.
  \item[(ii)]
  For any $ \mathbf{X}_0 \in \llbracket \mathbf{X}^{*}, \mathbf{Y}^{*}\rrbracket_{\mathcal{C}} \setminus \{ \mathbf{Y}^{*} \}$, all solutions $ \Phi \big( t; \mathbf{X}_0 \big)$ are attracted to $\mathbf{X}^{*}$, that is,  $\lim \limits_{t \to \infty} \Phi \big( t; \mathbf{X}_0 \big) = \mathbf{X}^{*}.$
  \item[(iii)]
  For any $ \mathbf{X}_0 \in \llbracket \mathbf{X}^{*}, \mathbf{Y}^{*}\rrbracket_{\mathcal{C}} \setminus \{ \mathbf{X}^{*} \}$, all solutions $ \Phi \big( t; \mathbf{X}_0 \big)$ are attracted to $\mathbf{Y}^{*}$, that is,  $\lim \limits_{t \to \infty} \Phi \big( t; \mathbf{X}_0 \big) = \mathbf{Y}^{*}.$
\end{description}
\end{theorem}
To adapt the hypotheses of this theorem to the system \eqref{sys} with $\mathcal{C}=\mathcal{K}$,~$\mathbf{X}^{*}=\mathbf{E}_N$, and $\mathbf{Y}^{*}=\mathbf{E}_W$, it must be shown that $\Phi \big(t, \llbracket \mathbf{E}_N, \mathbf{E}_W \rrbracket_{\mathcal{K}} \big)$ has compact closure in $\mathbb{R}_+^2$. In this context, it is instructive to note that any orbit
\begin{equation}
\label{orbit}
 \mathcal{O} (N,W):= \Big\{ \big( N(t), W(t) \big) \in \mathbb{R}_+^2: \ t \geq 0 \Big\}
 \end{equation}
of \eqref{sys} has a compact closure due to the existence of the absorbing set $\mathcal{X} \subset \mathbb{R}_+^2$ (given by \eqref{ab-set}) which, in effect, coincides with the order interval $ \llbracket \mathbf{E}_N, \mathbf{E}_W\rrbracket_{\mathcal{K}} :=[0,N_{\sharp}] \times [0,W_{\sharp}]$.

It is clear that direct application of Theorem \ref{teor1} to the order interval $\llbracket \mathbf{E}_N, \mathbf{E}_W \rrbracket_{\mathcal{K}}$ reaffirms, via item (i), the existence of $\mathbf{E}_c=(N_c,W_c)$ such that $\mathbf{E}_N <_{\mathcal{K}} \mathbf{E}_c <_{\mathcal{K}} \mathbf{E}_W$. On the other hand, the following proposition relates the two order subintervals of $\llbracket \mathbf{E}_N, \mathbf{E}_W \rrbracket_{\mathcal{K}}$ with the two basins of attraction $\mathcal{B}_{\mathbf{E}_N}$ and $\mathcal{B}_{\mathbf{E}_W}$ of the boundary equilibria $\mathbf{E}_N$ and $\mathbf{E}_W$ defined by \eqref{basins}.

\begin{figure}[t]
\begin{center}
\scalebox{1}{
\includegraphics{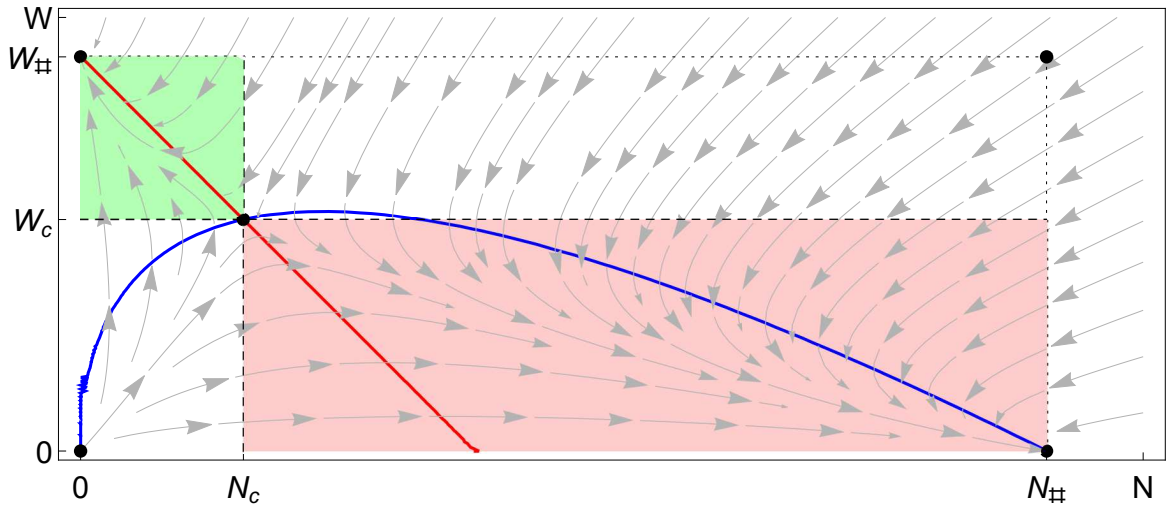}}
\end{center}
\caption{Phase portrait of the dynamical system \eqref{sys} with four equilibria $\textbf{E}_0, \textbf{E}_N, \textbf{E}_W,$ and $\textbf{E}_c,$ joined by  $N$-nullcline (blue-colored curve) and $W$-nullcline (red-colored curve) and the attraction regions of boundary equilibria $\textbf{E}_N, \textbf{E}_W$ defined by the two order intervals in Proposition \ref{prop4}.}
\label{fig-phase1}
\end{figure}

\begin{proposition}
\label{prop4}
For the system \eqref{sys}, it is fulfilled that
\[ \llbracket \mathbf{E}_N, \mathbf{E}_c\llbracket_{\mathcal{K}}  \subset \mathcal{B}_{\mathbf{E}_N}, \qquad  \rrbracket \mathbf{E}_c, \mathbf{E}_W\rrbracket_{\mathcal{K}}  \subset \mathcal{B}_{\mathbf{E}_W}, \]
where
\[ \llbracket \mathbf{E}_N, \mathbf{E}_c \llbracket_{\mathcal{K}} := \llbracket \mathbf{E}_N, \mathbf{E}_c \rrbracket_{\mathcal{K}} \setminus \{ \mathbf{E}_c \}, \qquad \rrbracket \mathbf{E}_c, \mathbf{E}_W\rrbracket_{\mathcal{K}} := \llbracket \mathbf{E}_c, \mathbf{E}_W \rrbracket_{\mathcal{K}} \setminus \{ \mathbf{E}_c \}. \]
\end{proposition}
\begin{proof}
Dynamical system \eqref{sys} fulfills the hypotheses of  Theorem \ref{teor1}, which will be applied to the order intervals $\llbracket \mathbf{E}_N, \mathbf{E}_c \rrbracket_{\mathcal{K}}$ and $\llbracket \mathbf{E}_c, \mathbf{E}_W \rrbracket_{\mathcal{K}}$, sets included in $ \llbracket \mathbf{E}_N, \mathbf{E}_W \rrbracket_{\mathcal{K}}$ where the flow is SOP (see Lemma \ref{lemma:sop-interval}). It is immediate to check that there is no equilibrium point inside the order interval $\llbracket \mathbf{E}_N, \mathbf{E}_c \rrbracket_{\mathcal{K}}$, and $\mathbf{E}_N$ is an attractor. By virtue of item (ii) of Theorem \ref{teor1}, all orbits $\mathcal{O} (N,W)$ started in $\llbracket \mathbf{E}_N, \mathbf{E}_c\llbracket_{\mathcal{K}}$ are attracted to $\mathbf{E}_N$. Therefore, $\llbracket \mathbf{E}_N, \mathbf{E}_c \llbracket_{\mathcal{K}}$ belongs to $\mathcal{B}_{\mathbf{E}_N}$. A similar rationale applies to the order interval $\rrbracket \mathbf{E}_c, \mathbf{E}_W \rrbracket_{\mathcal{K}}$ making use of item (iii) of  Theorem \ref{teor1}.
\end{proof}

Figure \ref{fig-phase1} displays order intervals $\llbracket \mathbf{E}_N, \mathbf{E}_c \rrbracket_{\mathcal{K}}$ and $\llbracket \mathbf{E}_c, \mathbf{E}_W \rrbracket_{\mathcal{K}}$ as pink and green rectangles, respectively. The pink rectangle belongs to $\mathcal{B}_{\mathbf{E}_N} \cap \mathcal{X}$ while the green one lies inside $\mathcal{B}_{\mathbf{E}_W} \cap \mathcal{X}$, that is, each order subinterval is included in  the intersection of the underlying basin of attraction and the absorbing set $\mathcal{X}$. The unmarked areas of $\mathcal{X}$ contain points $\mathbf{X}=(N,W)$ that may belong to either $\mathcal{B}_{\mathbf{E}_N}$ or $\mathcal{B}_{\mathbf{E}_W}$.

The following result establishes further properties of the reduced system \eqref{sys} related to its bistability.

\begin{proposition}
\label{prop5}
Dynamical system \eqref{sys} exhibits the saddle-point behavior and there exists an invariant threshold manifold that passes through the positive steady state $\mathbf{E}_c$ separating the attraction basins $\mathcal{B}_{\mathbf{E}_N}$ and $\mathcal{B}_{\mathbf{E}_W}$.
\end{proposition}
\begin{proof}
Notably, according to \cite{Jiang2004}, a dynamical system is said to admit a ``saddle-point behavior'' if it possesses two locally stable equilibria on the boundary of the state domain and one unstable (saddle-point) equilibrium in the interior of the state domain. Furthermore, the state domain of the system can be divided into three disjoint and invariant parts: two attraction basins (each containing one stable equilibrium on the boundary) and the so-called ``threshold'' manifold containing the unstable equilibrium that separates the attraction basins of two locally stable equilibria. Such a manifold is also referred to as the \emph{separatrix} of two attraction basins.

The saddle-point behavior of the system \eqref{sys} will be shown by applying the result summarized by H. Smith \cite[Theorem 3.2]{Smith2017}. For that purpose, we establish the cogency of four necessary hypotheses:
\begin{enumerate}[(H1)]
  \item
  The semiflow $\Phi(t,\cdot)$ generated by the system \eqref{sys} is strictly order-preserving on $\mathbb{R}^2_+$ with respect to $<_{\mathcal{K}}$ and order-compact for each $t>0$.
  \item
  The origin $\mathbf{E}_0=(0,0)$ is a repelling equilibrium.
  \item
  All orbits originated on the boundaries of $\mathbb{R}_+^2$ are confined to these boundaries.
  \item
  If $\mathbf{X}, \mathbf{Y} \in \mathbb{R}_+^2$ satisfy $\mathbf{X} <_{\mathcal{K}} \mathbf{Y}$ and either $\mathbf{X}$ or $\mathbf{Y}$ belongs to Int $\mathbb{R}_+^2$, then $\Phi(t,\mathbf{X}) \ll_{\mathcal{K}} \Phi(t,\mathbf{Y})$ for $t>0$. If $\mathbf{X}=(X_1,X_2) \in \mathbb{R}_+^2$ satisfies $X_i \neq 0, i=1,2,$ then $\Phi(t,\mathbf{X}) \in$ Int $\mathbb{R}_+^2$ for $t>0$.
\end{enumerate}

To show the validity of (H1), we recall the statement of Proposition \ref{prop-SOP} according to which the semiflow $\Phi(t,\cdot)$ generated by the system \eqref{sys} is SOP in the interior of  $\mathbb{R}_+^2$ and, therefore, it is also strictly order-preserving in Int $\mathbb{R}_+^2$. Strict monotonicity of $\Phi(t,\cdot)$ on the borders of $\mathbb{R}_+^2$ (which are, in fact, invariant sets $\Omega_N$ and $\Omega_W$, see Remark \ref{rem1}) follows from strict monotonicity of solutions of logistic equations \eqref{log-N} and \eqref{log-W}. Furthermore, the semiflow $\Phi(t,\cdot)$ generated by the system \eqref{sys} is order-compact since any orbit \eqref{orbit} of \eqref{sys} has a compact closure due to the existence of the absorbing set \eqref{ab-set}.

The hypothesis (H2) is cogent by virtue of Proposition \ref{prop2}, and the hypothesis (H3) is justified by invariance of the boundaries (sets  $\Omega_N$ and $\Omega_W$, see Remark \ref{rem1}).

Finally, the hypothesis (H4) is corroborated by the positiveness of the system trajectories engendered by positive initial conditions (see Proposition \ref{prop1}) along with the SOP property of the semiflow $\Phi(t,\cdot)$ on Int $\mathbb{R}_+^2$ (see Proposition \ref{prop-SOP}).

With the hypotheses (H1)-(H4) in force, we can now apply the result recaptured by H. Smith \cite[Theorem 3.2]{Smith2017} which basically affirms the following. If there is a unique equilibrium in $\mathbf{E}_c \in \rrbracket \mathbf{E}_N, \mathbf{E}_W \llbracket_{\mathcal{K}} \: \cap \; \text{Int }\mathbb{R}_+^2$ and it is a saddle point, then there exists an unordered positively invariant set
\[ \mathcal{S}_{\mathbf{E}_c} := \mathbb{R}_+^2 \setminus \big( \mathcal{B}_{\mathbf{E}_N} \cup \mathcal{B}_{\mathbf{E}_W} \big) \]
that contains the unstable equilibria $\mathbf{E}_0, \mathbf{E}_c$ and consists of points $\mathbf{X}_s \in \text{Int }\mathbb{R}_+^2$  such that
\[ \lim \limits_{t \to \infty} \Phi \big(t, \mathbf{X}_s \big) = \mathbf{E}_c \quad \text{for all} \;\; \mathbf{X}_s \in \mathcal{S}_{\mathbf{E}_c}.
\]
This completes the proof of Proposition \ref{prop5} and establishes the saddle-point behavior of the system \eqref{sys}.
\end{proof}

\section{Practical applications and final remarks}
\label{sec-appl}

The thorough analysis of the reduced model \eqref{sys} performed in Sections \ref{sec-analisys} and \ref{sec-monot} provides very useful insights for practical applications. By recalling the Stable Manifold Theorem (see, e.g., \cite[p. 107]{Perko2013}), it can be concluded that the invariant ``threshold'' manifold $\mathcal{S}_{\mathbf{E}_c}$ is, in effect, the stable invariant manifold of the saddle point $\mathbf{E}_c$ that is tangent to the eigenvector generated by the negative eigenvalue of $\mathbb{J}(\mathbf{E}_c)$. Furthermore, there also exists the unstable invariant manifold of the saddle point $\mathbf{E}_c$ that is tangent to the eigenvector generated by the positive eigenvalue of $\mathbb{J}(\mathbf{E}_c)$. The unstable manifold contains two monotone heteroclinic orbits \cite{Smith1995} that connect the saddle point  $\mathbf{E}_c$ with two local attractors  $\mathbf{E}_N$ and  $\mathbf{E}_W$.

\begin{figure}[t]
\begin{center}
 \includegraphics[width=7cm]{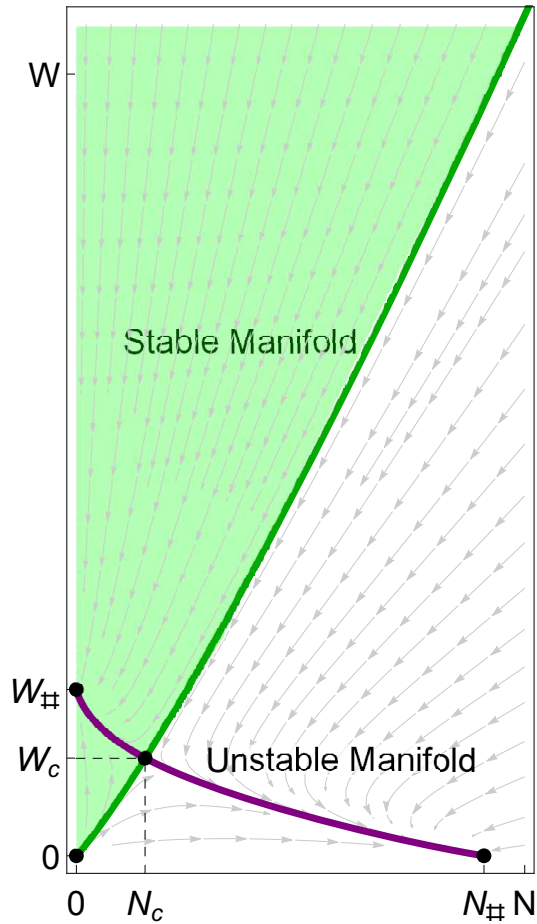}
\end{center}
\caption{Separatrix $\mathcal{S}_{\mathbf{E}_c}$ or the stable manifold of $\mathbf{E}_c$ (green curve) that separates the attraction basins $\mathcal{B}_{\mathbf{E}_N}$ and $\mathcal{B}_{\mathbf{E}_W}$ (green-shaded region) and the unstable manifold of $\mathbf{E}_c$ (purple curve) that connects three equilibria $\textbf{E}_N \leftarrow \textbf{E}_c \rightarrow  \textbf{E}_W$.}
\label{fig-phase2}
\end{figure}

Figure \ref{fig-phase2} displays the plots of the stable and unstable manifolds of  $\mathbf{E}_c$ (green and purple curves, respectively). It is clearly shown that the green curve plays the role of separatrix $\mathcal{S}_{\mathbf{E}_c}$ that divides $\mathbb{R}_+^2$ into two attraction basins $\mathcal{B}_{\mathbf{E}_N}$ (unshaded area) and $\mathcal{B}_{\mathbf{E}_W}$ (green-shaded region).

From the biological standpoint,  points in $ \mathcal{S}_{\mathbf{E}_c}$ indicate \emph{minimal viable population sizes} of wild and \textit{Wolbachia}-carrying mosquito populations that are related to the frequency-dependent Allee effect. In fact, if the initial  wild mosquito population size is $N_0 \ge 0$, one can compute the minimal viable population size  $\hat W_0=\hat W_0(N_0)$, such that $(N_0, \hat W_0) \in \mathcal{S}_{\mathbf{E}_c}$. Thus, if the initial  \textit{Wolbachia}-carrying mosquito population $W_0$ is lower than the minimal viable population size  $\hat W_0$, then  $(N_0,W_0) \in \mathcal{B}_{\mathbf{E}_N}$ and the underlying solution $\big( N(t), W(t) \big)$ of the system \eqref{sys} is attracted to $\mathbf{E}_N=(N_{\sharp},0)$, meaning the ultimate persistence of the wild population and progressive extinction of the \textit{Wolbachia}-infected population.

Similarly, if the initial condition $(N_0,W_0)$ assigned to the system \eqref{sys} lies \emph{above} $\mathcal{S}_{\mathbf{E}_c}$, that is, $W_0 > \hat W_0$, it implies the initial size of the \textit{Wolbachia}-infected population $W_0$ exceeds its minimal viable population size $\hat{W}_0$. In such a case, we have that  $(N_0,W_0) \in \mathcal{B}_{\mathbf{E}_W}$ and the underlying solution $\big( N(t), W(t) \big)$ of the system \eqref{sys} is attracted to the desired equilibrium $\mathbf{E}_W=(0, W_{\sharp})$, meaning the ultimate persistence of the \textit{Wolbachia}-infected population and progressive extinction of the wild mosquito population.

\begin{table}[t]
\centering
\begin{tabular}{lll}
  \hline \\  [-0.5ex]
  \textbf{Description} & \textbf{Assumed value} &\textbf{References}  \\ [1.5ex]
  \hline \\ [-1.5ex]
Fecundity rate of uninfected insects & $\rho_N = 4.55$                         & \cite{Campo2018,Styer2007} \\
Fecundity rate of infected insects & $\rho_W = 0.5 \times \rho_N= 2.27$      & \cite{Campo2018,Dorigatti2018,McMeniman2009,McMeniman2010} \\
Natural mortality rate of uninfected insects & $\alpha_N = 0.03333$                    & \cite{Campo2018,Styer2007} \\
Natural mortality rate of infected insects& $\alpha_W = 2 \times \alpha_N =0.06666$ & \cite{Campo2018,Dorigatti2018,McMeniman2009,McMeniman2010}\\
Competition parameter  of uninfected insects & $\beta_N = 2.61258 \times 10^{-3}$      & fitted using data from \cite{Crain2011,Oliveira2017,Suh2013} \\
Competition parameter  of infected insects& $\beta_W = 3.12792 \times 10^{-3}$      & fitted using data from \cite{Crain2011,Oliveira2017,Suh2013} \\ [0.5ex]
  \hline
\end{tabular}
\caption{Parameter values for the reduced model \eqref{sys} corresponding to \textit{Aedes aegypti} mosquitoes and the \textit{wMelPop} strain of \textit{Wolbachia}.} \label{tab2}
\end{table}

To plot all figures presented in this paper, we have used the parameter values given in Table \ref{tab2}. These values correspond to the \textit{wMelPop} strain of \textit{Wolbachia}, which is regarded as the best one for controlling dengue infections among human individuals since it confers the most profound resistance to the replication of dengue virus in mosquitoes \cite{Dorigatti2018,Woolfit2013}. However, many scholars point out that the \textit{wMelPop} strain is associated with high ``fitness cost'' since it reduces female fecundity, the viability of eggs, and the lifespan of infected mosquitoes \cite{Dorigatti2018,McMeniman2010,Ritchie2015}. Therefore, the successful invasion of mosquitoes carrying the \textit{wMelPop} strain of \textit{Wolbachia} and  their durable persistence even in small detached localities appears to be a challenging task.

It is widely known that female mosquitoes are major transmitters of dengue and other vector-borne infections. When deliberately infected with \textit{Wolbachia}, they lose their vector competence by becoming far less capable of developing a viral load sufficient for transmission of the virus to human individuals. Due to this remarkable feature, \textit{Wolbachia}-based biocontrol of mosquito populations has recently emerged as a novel method for the prevention and control of vector-borne infections.

The ultimate goal of \textit{Wolbachia}-based biocontrol consists of seeking the eventual elimination of wild insects (capable of transmitting the virus to human individuals) by performing periodic releases of \textit{Wolbachia}-carrying mosquitoes in some determined localities initially populated by wild mosquitoes. The practical implementation of this method requires to mass-rear a massive quantity of \textit{Wolbachia}-infected insects for posterior releases, and the desired result is propelled by the progressive \textit{Wolbachia} invasion and its durable establishment in wild mosquito populations. The final outcome of this process is usually referred to as ``population replacement''.

Let us now provide some useful insights and practical interpretations derived from Figure \ref{fig-phase2} while keeping in mind the primary goal of \textit{Wolbachia}-based biocontrol.  First, we recall that the attraction basin $\mathcal{B}_{\mathbf{E}_W}$ contains the initial conditions $\big( N(0),W(0) \big)$ starting from which the trajectories of the system \eqref{sys} converge to the desired boundary equilibrium $\mathbf{E}_W=(0,W_{\sharp})$ and the population replacement will be eventually achieved.

Suppose now that the population replacement is sought to be achieved with a single (or \emph{inundative}) initial release of \textit{Wolbachia}-carrying insects. To determine the size of such an abundant release, the current size of the wild mosquito population should be first assessed by some known technique \cite{Cianci2013,Gouagna2015}. Once the abundance of wild mosquitoes $N_0$ is fairly estimated, the information regarding the minimal viable sizes $\hat W_0=\hat W_0(N_0)$ such that $\big( N_0, \hat{W}_0 \big) \in \mathcal{S}_{\mathbf{E}_c} $ of \textit{Wolbachia}-carrying mosquito populations will be of the utmost importance, and this information is explicitly supplied by model \eqref{sys} and its underlying parameters.

Upon closer inspection of Figure \ref{fig-phase2}, we observe that the basin of attraction $\mathcal{B}_{\mathbf{E}_W}$ (green-shaded region) is much smaller than  $\mathcal{B}_{\mathbf{E}_N}$ (unshaded region). Therefore, when the size of the wild mosquito population is close to saturation or its carrying capacity $N_{\sharp}$, an extreme amount of \textit{Wolbachia}-carriers will be necessary for a single inundative release.

On the other hand, it is instructive to recall that wild mosquito populations may exhibit seasonal size variations \cite{Delatte2009,Gouagna2015}. In this context, the best timing for \textit{Wolbachia}-based biocontrol by a single inundative release will be the period of relatively low mosquito abundance. Such periods usually correlate with cooler and windier seasons in tropical and subtropical regions \cite{Gouagna2015} or arise after carefully planned and thoroughly implemented vector control measures \cite{Pliego2020}.

For different initial sizes $N_0 = \lambda N_{\sharp}$ of wild mosquito population expressed as fractions $\lambda \in \{0.25, 0.5, 0.75, 1\}$ of $N_{\sharp}$, one can estimate the corresponding minimal release sizes of \textit{Wolbachia}-carrying insects $\hat{W}_0=\hat{W}_0(\lambda N_{\sharp}) = \hat{\lambda} N_{\sharp}$ (also expressed as the multiplicatives of $N_{\sharp}$) that ensure the population replacement by a single inundative release. The corresponding values of $\lambda$ and $\hat{\lambda}$ are presented in Table \ref{table:oneshoot} (considering the parameters of Table \ref{tab2}), and the points $\Big( \lambda N_{\sharp}, \hat{\lambda} N_{\sharp} \Big) \in \mathcal{S}_{\mathbf{E}_c}$ are displayed in Figure \ref{fig-minWB} in red color.

\begin{figure}[t]
\begin{center}
 \includegraphics[width=7cm]{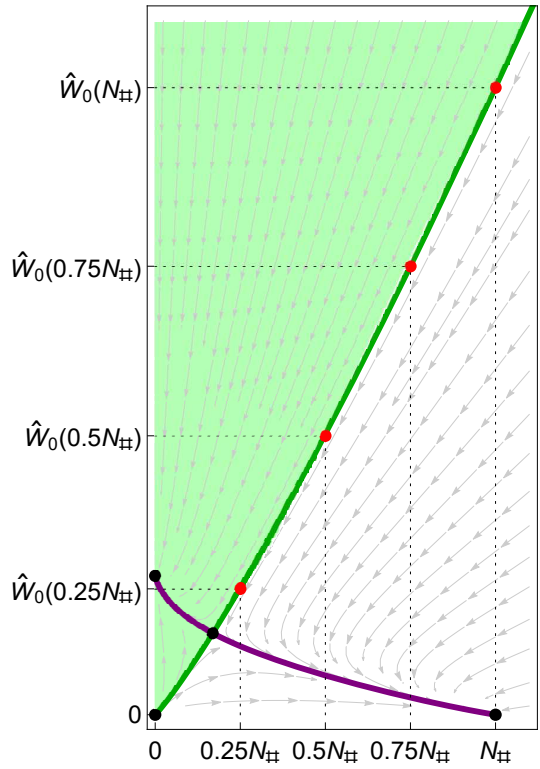}
\end{center}
\caption{Minimal \textit{Wolbachia}-infected population sizes $\hat W_0= \hat W_0 \big(\lambda N_{\sharp} \big)$ for different initial sizes $N_0=\lambda N_{\sharp}, \lambda \in \{0.25, 0.5, 0.75, 1\}$  of the wild mosquito population.}
\label{fig-minWB}
\end{figure}

\begin{table}[t]
\centering
\begin{tabular}{cc}
  \hline \\  [-0.5ex]
  \textbf{$\lambda$ (Initial wild population $N_0=\lambda N_{\sharp}$)} & \textbf{$\hat \lambda$ (Minimum viable population $\hat W_0= \hat {W}_0 \big( \lambda N_{\sharp} \big) = \hat{\lambda} N_{\sharp}$}   \\ [1.5ex]
  \hline \\ [-1.5ex]
0.25 & 0.38 \\
0.5 & 0.83 \\
0.75 & 1.32 \\
1 & 1.85 \\
  \hline
\end{tabular}
\caption{Minimal \textit{Wolbachia}-infected population size $\hat W_0= \hat W_0 \big(\lambda N_{\sharp} \big)$ for different initial sizes $N_0=\lambda N_{\sharp}, \lambda \in \{0.25, 0.5, 0.75, 1\}$  of the wild mosquito population, see Figure \ref{fig-minWB}.} \label{table:oneshoot}
\end{table}

From Table \ref{table:oneshoot} and Figure \ref{fig-minWB}, we observe that the minimal \textit{Wolbachia}-infected population sizes needed for single inundative releases always exceed the initial sizes of the wild mosquito population. When the initial size of the wild population is higher than $0.5 N_{\sharp}$, the population replacement becomes substantially more difficult to reach by a single inundative release since it requires to mass-rear en masse a vast quantity of \textit{Wolbachia}-carriers.

As an alternative to a single inundative release, one may perform several periodic (or \emph{inoculative}) releases. This strategy may seem reasonable if the  mass-rearing facility cannot produce the vast quantity $\hat{\lambda} N_{\sharp}$ of \textit{Wolbachia}-carriers en masse but is capable of producing a smaller quantity of \textit{Wolbachia}-carrying insects every $\tau$ days. However, under such a setting, the total amount of \textit{Wolbachia}-infected mosquitoes necessary to reach the population replacement will be larger than in the case of a single inundative release.

From the mathematical standpoint, periodic inoculative releases of \textit{Wolbachia}-carrying insects can be modeled by the following impulsive dynamical system:
\begin{subequations}
\label{sys-imp}
\begin{align}[left = \empheqlbrace\,]
\label{sys-imp-N}
& \frac{d N}{dt} = \rho_N N \left(\frac{N}{N+W} \right) - \alpha_N N - \beta_N N(N+W), & & N(0)=N_0 \\
\label{sys-imp-W}
& \frac{d W}{dt} = \rho_W W - \alpha_W W - \beta_W W(N+W), & & W(0)= \Lambda, \;\;\; W(i \tau^+) = W(i\tau^-) + \Lambda, \ i=1, 2, \ldots, n-1
\end{align}
\end{subequations}
where $\tau$ denotes the period of releases, $\Lambda=\text{const}$ stands the release size, and $n$ defines the number of releases. Notably, $W(i\tau^{\pm})$ denote the right and left limits of the function $W(t)$ at $t=i\tau$. Formal analysis of the impulsive system \eqref{sys-imp} is a challenging task and may be proposed as an object for further studies. Therefore, in the context of this paper, we limit ourselves to revising its numerical solutions in order to assess the practical value of periodic releases and to compare their overall performance with an outcome of a single inundative release.

By performing a series of numerical simulations, we have estimated the minimal release sizes $\Lambda:= \hat{\lambda} N_{\sharp}$ (also expressed as the multiplicatives of $N_{\sharp}$) for $\tau=1$ day and $\tau=3$ days and taking different initial sizes of wild mosquito population $N_0$, expressed as fractions $\lambda \in \{0.25, 0.5, 0.75, 1\}$ of $N_{\sharp}$.

\begin{table}[t]
\centering
\begin{tabular}{cccc}
  \hline \\  [-0.5ex]
  \textbf{$\lambda$ ($N_0=\lambda N_{\sharp}$)} & \textbf{$\hat \lambda$ ($\Lambda= \hat{\lambda} N_{\sharp},$ release size)}  &  \textbf{Period of releases ($\tau$ days)} & \textbf{Number of releases, $n$}  \\ [1.5ex]
  \hline \\ [-1.5ex]
0.25 & 0.25 & 1 & 5 \\
0.25 & 0.3615 & 3 & 3 \\
0.5 & 0.39 &  1 & 9 \\
0.5 & 0.773 &  3 & 3 \\
0.75 & 0.43 &  1 & 11 \\
0.75 & 1.178 &  3 & 4 \\
1 & 0.43 & 1 & 12 \\
1 & 1.39 & 3 & 8 \\
  \hline
\end{tabular}
\caption{For different initial sizes of the wild mosquito population (column 1), the release size of the \textit{Wolbachia}-carrying population (column 2) necessary to ensure the population replacement by inoculative releases each $\tau$ days (column 3) and the number of releases (column 4) are indicated.} \label{table:mshoot}
\end{table}

In Table  \ref{table:mshoot}, we present minimum release sizes $\Lambda$ that ensure the population replacement by $n$ periodic releases even when the initial condition $\big( N_0, \Lambda \big)$ lies inside the attraction basin $\mathcal{B}_{\mathbf{E}_N}$ of the boundary equilibrium $\mathbf{E}_N=(N_{\sharp},0)$ (that is, strictly below the separatrix $\mathcal{S}_{\mathbf{E}_c}$).  Revising the entries of Table \ref{table:mshoot}, it is easy to detect several patterns or ``tradeoffs'' between the frequency of releases $\tau$, constant release size $\Lambda$, and the overall number of releases $n$ needed to ensure the population replacement. Namely, more frequent releases ($\tau=1$) require smaller release sizes $\Lambda$ and shorter overall time of the release program but a greater number of releases $n$.  The latter is quite logical and not only aligns with common sense but also bears similarities with other works dealing with periodic releases of mosquitoes \cite{Bliman2019}. In this context, the anticipated knowledge of the production costs related to the mass-rearing of \textit{Wolbachia}-carrying insects and the logistics costs for performing field releases are important for choosing the release frequency. Although we have no reliable information regarding such costs, the impulsive system \eqref{sys-imp} may serve to be of potential utility in the future when healthcare entities eventually decide to evaluate this method of biological vector control.

On the other hand, the outcomes of numerical simulations performed on the original model \eqref{sys} can be also compared with those obtained for the impulsive system \eqref{sys-imp}. Contrasting the values of $\hat{\lambda}$ from Tables \ref{table:oneshoot} and \ref{table:mshoot} for the same values of $N_0=\lambda N_{\sharp}$, we observe that they bear a more striking difference for $\tau=1$ than for $\tau=3$. Moreover, the mentioned difference is smaller for the smaller values of $N_0$ (such as $0.25 N_{\sharp}$ and $0.5 N_{\sharp}$) and becomes more noticeable for the larger values of $N_0$ (such as $0.75 N_{\sharp}$ and $N_{\sharp}$). Thus, the release programs based on periodic inoculative releases seem more practicable when the initial size of the wild mosquito population is close to its saturation level $N_{\sharp}$.

\begin{figure}[t!]
\begin{center}
\begin{tabular}{cc}
 \includegraphics[width=0.45\textwidth]{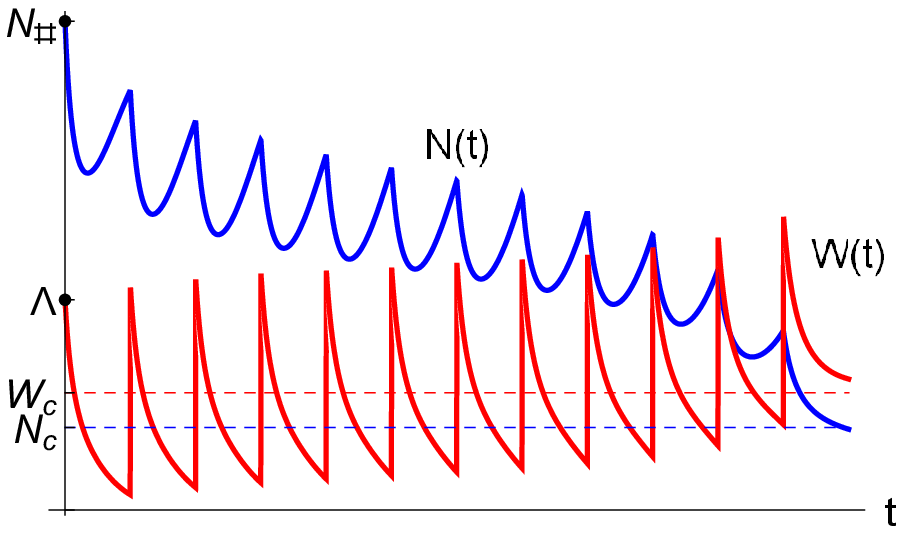} &  \includegraphics[width=0.45\textwidth]{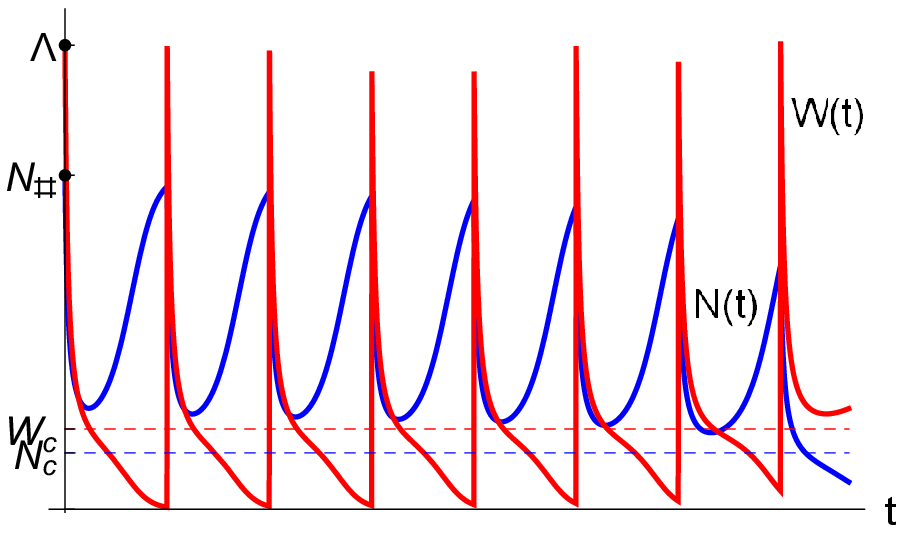} \\
 & \\
$N_0= N_{\sharp}, W_0 =0.43 N_{\sharp}$ (release size), $\tau=1, n=12$ & $N_0= N_{\sharp}, W_0 =1.39 N_{\sharp}$ (release size), $\tau=3, n=8$ \\
 & \\
\includegraphics[width=0.4\textwidth]{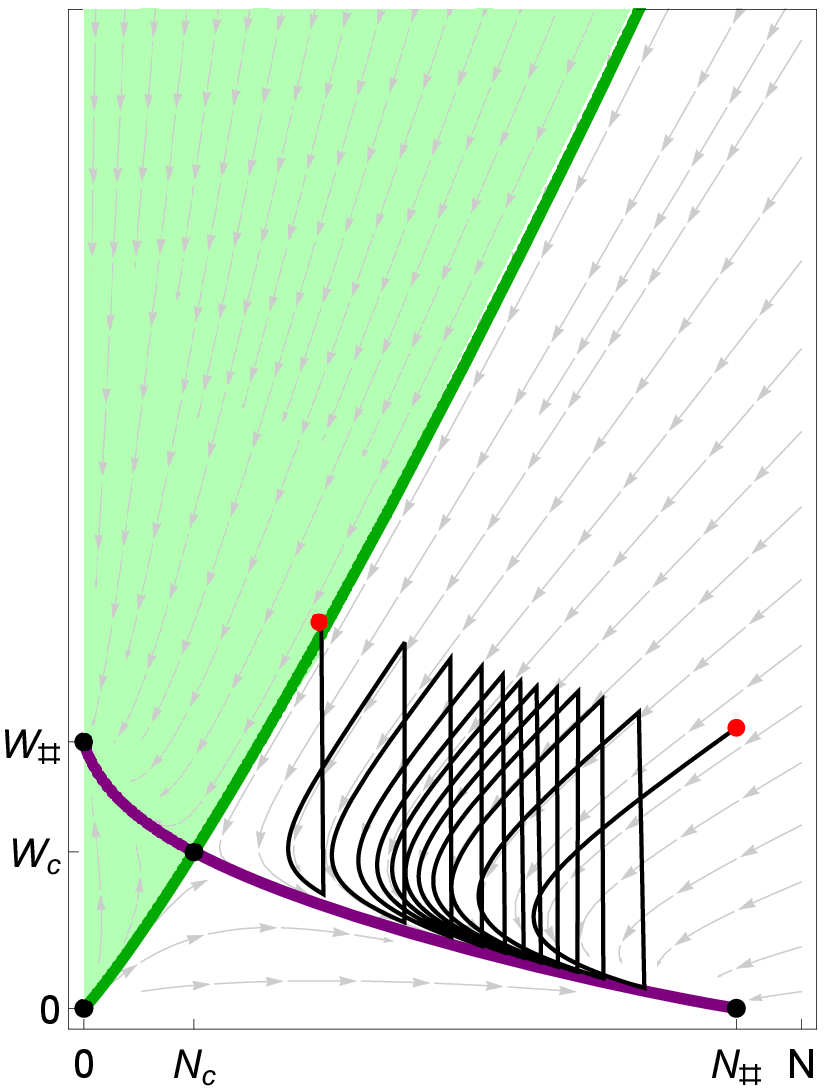}  &  \includegraphics[width=0.4\textwidth]{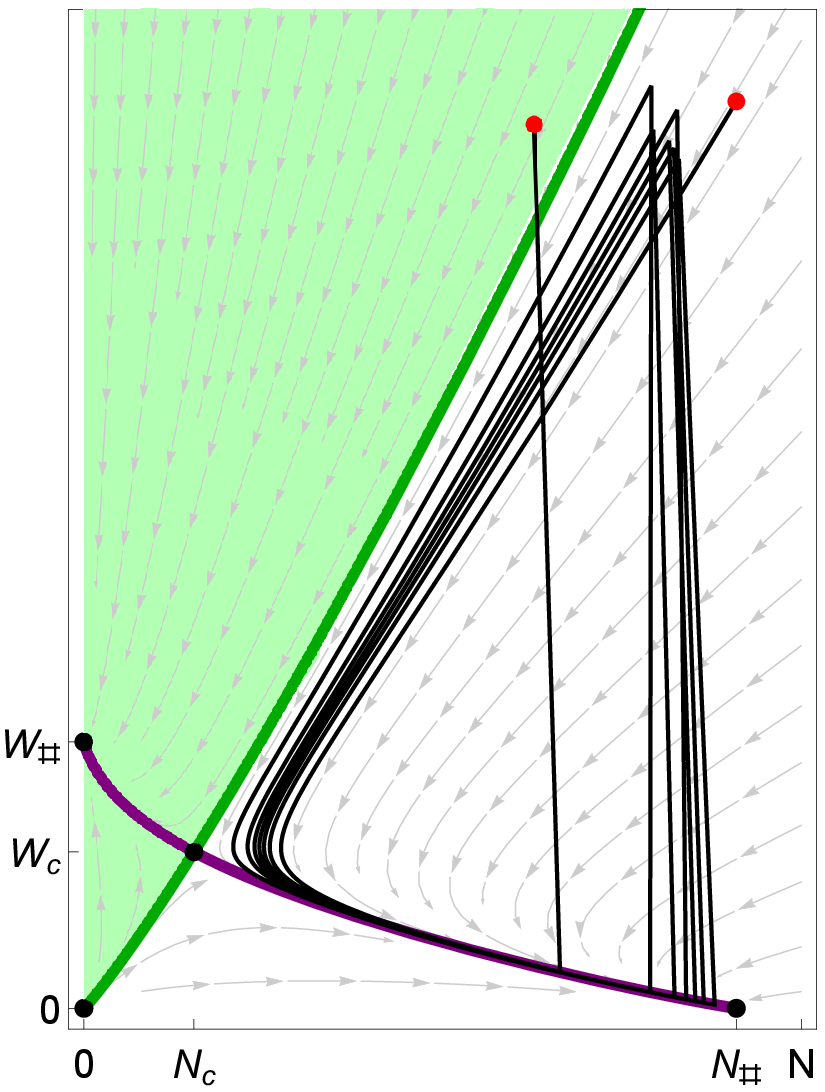} \\
 \end{tabular}
\end{center}
\caption{Trajectories $N(t), W(t)$ \emph{(upper row)} and the orbits in the phase space \emph{(lower row)} of the impulsive system \eqref{sys-imp} engendered by the initial conditions $N_0=N_{\sharp},  W_0= \Lambda =\hat{\lambda} N_{\sharp}$ corresponding to periodic releases with $\tau=1$ \emph{(left column}) and $\tau=3$ \emph{(right column)} for constant release sizes $\Lambda=\hat{\lambda} N_{\sharp}$ given in the last two rows of Table \ref{table:mshoot}.}
\label{fig-per}
\end{figure}

Figure \ref{fig-per} displays simulation results for the impulsive system \eqref{sys-imp} when the wild mosquito population is at saturation $N_0=N_{\sharp}$. The left column of Figure \ref{fig-per} corresponds to daily releases ($\tau=1$ day) and the right one corresponds to inoculative releases performed every three days ($\tau=3$). The upper charts of Figure \ref{fig-per} present the system's trajectories  $N(t)$ and $W(t)$ drawn by blue- and red-colored curves, respectively, and also bear two dashed lines marking the coordinates of the coexistence equilibrium $\mathbf{E}_c=\big(N_c,W_c)$. The lower charts exhibit the underlying parts of  orbits $\mathcal{O}(N,W) = \Big\{ \big( N(t),W(t) \big) \in \mathbb{R}_+^{2}: \ t \geq 0  \Big\}$ in the phase space that start in $(N_0,W_0) \in \mathcal{B}_{\mathbf{E}_N}$ (red-colored point below the separatrix $\mathcal{S}_{\mathbf{E}_c}$) and move the system states to the attraction basin $\mathcal{B}_{\mathbf{E}_W}$ of the desired boundary equilibrium $\mathbf{E}_W=(0,W_{\sharp})$ (red-colored point above the separatrix $\mathcal{S}_{\mathbf{E}_c}$).

The periodic inoculative releases are suspended when the orbit $\mathcal{O}(N,W)$ of \eqref{sys-imp} crosses the separatrix $\mathcal{S}_{\mathbf{E}_c}$ and enters the attraction basin $\mathcal{B}_{\mathbf{E}_W}$ (green-shaded region in the lower charts of Figure \ref{fig-per}). The latter is also clearly visible in the upper charts of Figure \ref{fig-per}: the system trajectory $N(t)$ decays and crosses the blue-colored dashed line, while the trajectory $W(t)$ remains strictly above the red-colored dashed line.

Browsing once again the simulation results given in Tables \ref{table:oneshoot} and \ref{table:mshoot} and contrasting them for each particular value of $N_0= \lambda N_{\sharp}, \lambda \in \{ 0.25,0.5,0.75, 1 \}$, we may conclude that, from the practical standpoint, an implementation of a single inundative release seems more rational and operative than several periodic inoculative releases. Effectively, under the ``worst scenario'', i.e., when $N_0=N_{\sharp}$ (this situation is illustrated in Figure \ref{fig-per}) it is necessary to mass-rear at least $5.16 N_{\sharp}$ of \textit{Wolbachia}-carriers during 12 days (with $\tau=1$) or at least $11.12 N_{\sharp}$ of \textit{Wolbachia}-carriers during 18 days (when $\tau=3$), while a single inundative release only requires to mass-rear $1.85 N_{\sharp}$ of \textit{Wolbachia}-infected insects, albeit all at once.

Thus, the reduced bidimensional model \eqref{sys} has resulted in a quite handy and easily interpretable tool for determining the appropriate size of a single inundative release or periodic releases of \textit{Wolbachia}-carrying insects since it explicitly yields the dependence between minimal viable population sizes of wild and \textit{Wolbachia}-infected mosquito populations.

\section*{Acknowledgments}

Diego Vicencio was supported by the program CONICYT PFCHA/Doctorado Becas Chile/2017-21171813 and FONDECYT grant N 1200355 ANID-Chile program. Olga Vasilieva acknowledges financial support from the National Fund for Science, Technology, and Innovation (Autonomous Heritage Fund \emph{Francisco Jos\'e de Caldas}) by way of the Research Program No. 1106-852-69523 (Principal Investigator: Hector J. Martinez), Contract: CT FP 80740-439-2020 (Colombian Ministry of Science, Technology, and Innovation -- Minciencias), Grant ID: CI-71241 (Universidad del Valle, Colombia). Olga Vasilieva also appreciates the endorsement obtained from the STIC AmSud Program for regional cooperation (20-STIC-05 NEMBICA project, international coordinator: Pierre-Alexandre Bliman, INRIA -- France). Pedro Gajardo was partially supported by FONDECYT grant N 1200355 ANID-Chile program.

\bibliographystyle{plain}
\addcontentsline{toc}{section}{References}


\end{document}